\documentclass[12pt,a4paper,oneside]{amsart}
\usepackage{amsfonts, amsmath, amssymb, amsthm}
\usepackage{anysize}
\usepackage{mathtools}
\marginsize{2cm}{2cm}{1.5cm}{1.5cm}

\usepackage[utf8]{inputenc}
\usepackage[english]{babel}
\usepackage{cmap}
\usepackage{enumerate}
\usepackage{xcolor}
\usepackage{hyperref}
\hypersetup{
    colorlinks   = true, 
    urlcolor     = blue, 
    linkcolor    = blue, 
    citecolor    = red   
}

\newcommand{\Href}[2]{\hyperref[#2]{#1~\ref{#2}}}

\usepackage{cancel}

\newtheorem{thm}{Theorem}[section]
\newtheorem{prp}{Proposition}[section]
\newtheorem{lem}{Lemma}[section]
\newtheorem{cor}{Corollary}[section]

\newtheorem{claim}{Claim}[section]

\theoremstyle{definition}
\newtheorem{dfn}{Definition}[section]
\newtheorem{rem}[dfn]{Remark}

\newcommand{\st}{:\;}
\def\R{{\mathbb R}}%
\newcommand{\Red}{\R^d}

\newcommand{\hballf}{\hbar}

\providecommand{\parenth}[1]{\left(#1\right)}%
\providecommand{\braces}[1]{\left\{#1\right\}}%
\providecommand{\brackets}[1]{\left[#1\right]}%

\newcommand{\ball}[1]{\mathbf{B}^{#1}}
\newcommand{\norm}[1]{\left\|#1\right\|}
\newcommand{\enorm}[1]{\left|#1\right|}

\newcommand{\funpos}[1]{\mathcal{E}\! \left[ #1\right]}
\newcommand{\funppos}[1]{\mathcal{E}^{+}\! \left[ #1\right]}

\def\polar{\circ}
\newcommand{\polarset}[1]{{#1}^{\polar}}%

\newcommand{\iprod}[2]{\left\langle#1,#2\right\rangle}%
\newcommand{\id}{\mathrm{Id}}
\def\supp{\mathop{\rm supp}}%

\title{The John inclusion for log-concave functions}
\author{Grigory Ivanov}
\address{Pontif\'icia Universidade Cat\'olica do Rio de Janeiro \\
Departamento de Matem\'atica, \\
Rua Marqu\^es de S\~ao Vicente, 225 \\
Edif\'{\i}cio Cardeal Leme, sala 862, \\
22451-900 G\'{a}vea, Rio de Janeiro, Brazil}
\email{grimivanov@gmail.com}
\thanks{The author is supported by Projeto Paz and by CAPES (Coordena\c{c}\~ao de Aperfei\c{c}oamento de Pessoal de N\'ivel Superior) - Brasil, grant number 23038.015548/2016-06.}

\begin{document}
\begin{abstract}
John's inclusion states that a convex body in $\mathbb{R}^d$ can be covered by the $d$-dilation of its maximal volume ellipsoid. We obtain a certain John-type inclusion for log-concave functions. As a byproduct of our approach, we establish the following asymptotically tight inequality:
\\
\noindent
For any log-concave function $f$ with finite, positive integral, there exist a positive definite matrix $A$, a point $a \in \mathbb{R}^d$, and a positive constant $\alpha$ such that
\[
\chi_{\mathbf{B}^{d}}(x) \leq \alpha f\!\!\left(A(x-a)\right) \leq \sqrt{d+1} \cdot 
e^{-\frac{\left|x\right|}{d+2} + (d+1)},
\]
where $\chi_{\mathbf{B}^{d}}$ is the indicator function of the unit ball $\mathbf{B}^{d}$.
\end{abstract}
\maketitle

\section{Introduction}

The maximal volume ellipsoid contained within a given convex body, called the \emph{John ellipsoid}, is fundamental in modern convexity and asymptotic geometric analysis. Fritz John, in his seminal paper \cite{John}, derived the following property of the maximal volume ellipsoid, sometimes referred to as the \emph{weak John theorem} \cite{artstein2015asymptotic} or \emph{John's inclusion} :
\begin{prp}\label{prp:john_inclusion}
Let $K$ be a convex body in  $\Red$ whose John ellipsoid is the unit ball $\ball{d}.$  Then the following inclusion holds:
\begin{equation}\label{eq:johninclusion}
\ball{d} \subset K \subset d \cdot \ball{d}.
\end{equation}
\end{prp}

Recently, the notion of the John ellipsoid has been extended to the setting of logarithmically concave functions \cite{alonso2018john, ivanov2022functional, ivanovIMRN}. The generalization is rather straightforward. 
Instead of convex bodies, one considers upper semi-continuous log-concave functions of finite and positive integral, which will be called \emph{proper} log-concave functions.
The set of ellipsoids is the set of “affine positions” of the unit ball $\ball{d}.$ In the functional setting, one considers the  \emph{positions} of a given function $g$ 
on $\Red$, defined as
\[
\funpos{g} 
= 
\braces{\alpha\,g\!\parenth{A x + a}
\st 
A\in\R^{d\times d}\text{ non-singular},\;\alpha>0,\;a\in\Red}.
\]
Instead of set inclusions, one compares functions pointwise: 

We say that a function $f_1$ on $\Red$ is \emph{below} another function $f_2$ on $\Red$ (or that $f_2$ is \emph{above} $f_1$)
and denote it as $f_1 \leq f_2$ if $f_1$ is pointwise less than or equal to $f_2,$
that is, $f_1(x) \leq f_2(x)$ for all $x \in \Red.$  

We will call a \emph{John function of $f$ with  respect to a given function $w$} any solution to the  
 
 \medskip
 \noindent
\textbf{Functional John problem:} Find
\begin{equation}\label{eq:john_func_problem}
\max\limits_{g \in \funpos{w} } 
	\int_{\Red} g
	\quad \text{subject to} \quad
	g \leq f.
\end{equation}

The only remaining issue is to choose which function $w$ should be considered as the analogue of the unit ball. 
Following \cite{ivanov2022functional}, we will mostly use the \emph{height function} of $\ball{d+1}$, defined as
\[
\hballf(x) = 
 \begin{cases}
\sqrt{1 - \enorm{x}^2},& 
\text{if } x \in \ball{d},\\
0,&\text{otherwise},
 \end{cases}
\]
as $w$ in \eqref{eq:john_func_problem}.

Even though many properties of solutions to \eqref{eq:john_func_problem} for various choices of $w$ have been understood, there has been no analogue of John's inclusion in the functional setting. 
Our goal is to correct this oversight.

 We believe that the root of the issue lies in the hidden polar duality in \eqref{eq:johninclusion}.
 We will elaborate on this in the next section.

Recall that the polar $\polarset{K}$ of a set $K$ in $\R^d$ is defined by  
\[
\polarset{K} = \braces{\,p \in \R^d \st \iprod{p}{x} \leq 1
\text{ for all } x \in K}\,.
\]  
Using the notion of polarity, one can re-write John's inclusion  \eqref{eq:johninclusion} in the following equivalent form:
\begin{equation}\label{eq:johninclusion_with_polars}
\ball{d} \subset K  
\quad 
\text{and} 
\quad
 \frac{\ball{d}}{d} \subset \polarset{K}.
\end{equation}

Interestingly, this form can be easily translated to the functional setting. 

Recall that the \emph{polar function} $\polarset{f}$ of a given non-negative function $f$ is defined as
\[
\polarset{f}(p) 
= 
\inf\limits_{\{x \st f(x) > 0\}} 
\frac{e^{-\iprod{p}{x}}}{f(x)}.
\]
 
The main result of this paper is the following John-type inclusion in the functional setting:

\begin{thm}\label{thm:john_inclusion_one_f}
Assume that an upper semi-continuous log-concave function $f \colon \Red \to [0, +\infty)$ of finite and positive integral is in a position such that $\hballf$ is a John function of $f$ with respect to $\hballf.$  Then 
\[
\hballf \;\leq\; f
\quad\text{and}\quad
e^{-(d+1)} \cdot \hballf\!\circ\!\brackets{(d+1)\,\id_d}\leq\; \polarset{f},
\]
where $\id_d$ is the identity on $\R^d.$
\end{thm}

One of the most basic inequalities related to log-concave functions states that a proper log-concave function is above some position of the indicator function $\chi_{\ball{d}}$ of the unit ball $\ball{d}$  and is below a position of the function $e^{-\enorm{x}}.$
As a byproduct of our method, we establish the following asymptotically optimal version of this basic inequality:
\begin{lem}\label{lem:flat_functions_inclusion}
Let $f$ be a proper log-concave function on $\R^d.$ There exist a positive definite matrix $A,$ a point $a \in \R^d,$ and a positive constant $\alpha$ such that
\[
\chi_{\ball{d}}(x) 
\;\leq\; 
\alpha\,f\!\parenth{A\parenth{x-a}} 
\;\leq\; 
\sqrt{d+1}\,\cdot 
e^{-\frac{\enorm{x}}{d+2} + (d+1)}.
\]
\end{lem}

A dual construction to the John ellipsoid is the so-called L\"owner ellipsoid, which is the minimal volume ellipsoid containing a given convex body. In the classical setting, the two ellipsoids are related by polar duality — the unit ball $\ball{d}$ is the John ellipsoid of $K$ if and only if $\ball{d}$ is the L\"owner ellipsoid of $\polarset{K}.$ 
This polarity property directly yields an inclusion for the L\"owner ellipsoid similar to that of 
\Href{Proposition}{prp:john_inclusion}.

The notion of the L\"owner ellipsoid can be extended to the setting of log-concave functions as well (see \cite{li2019loewner, ivanov2021functional, ivanovIMRN, lasserre2015generalization,alonso2022best}). However, we will show in \Href{Section}{sec:absence_lowner_inclusion} that there is no L\"owner-type inclusion for reasonable candidates for a L\"owner function.

\subsection{Notations}
The standard Euclidean unit ball in $\R^d$ is denoted by $\ball{d}.$
We identify the space $\R^d$ with the subspace of $\R^{d+1}$ consisting of vectors whose last coordinate vanishes. We use $[n]$ to denote $\{1, \dots, n\}$ for a natural $n.$  The support $\supp f$ of a non-negative function $f$ on $\R^d$ is the set on which the function is positive:
\[
\supp f = \braces{x \in \R^d \st f(x) > 0 }.
\]
The supremum norm of a bounded function $f$ on $\R^d$ is denoted by 
$\norm{f}_\infty.$

\section{John ellipsoid and John function}
In this section, we recall several useful properties of the John ellipsoid and John functions. 

\subsection{John ellipsoid and Duality}
\Href{Proposition}{prp:john_inclusion} follows from Fritz John's characterization \cite{John, ball1992ellipsoids} of the maximal volume ellipsoid within a convex body:

\begin{prp}\label{prp:johncond}
Let $K$ be a convex body in  $\Red$ containing the unit ball  $\ball{d}.$  Then the following assertions are equivalent:
\begin{enumerate}
 \item
$\ball{d}$ is the John ellipsoid of $K.$
 \item
There are points 
${u}_1,\ldots,{u}_m$ on the boundaries of $\ball{d}$ and $K,$ and
positive weights $c_1,\ldots,c_m$ such that 
\begin{equation}\label{eq:johncond}
\sum\limits_{i \in [m]} c_i \,{u}_i\otimes {u}_i = \id_d 
\quad \text{and} \quad 
\sum\limits_{i \in [m]} c_i \,{u}_i = 0.
\end{equation}
\end{enumerate}
\end{prp}

We believe the core issue in obtaining a John-type inclusion in the functional setting lies in the standard identification of \(\R^d\) with its dual space \(\bigl(\R^d\bigl)^*\). Looking at John's condition \eqref{eq:johncond}, the operator \(u_i \otimes u_i\) can be written as \(u_i u_i^T\). Formally, \(u_i^T\) is a functional in the dual space \(\bigl(\R^d\bigl)^*\). Also, the polar set \(\polarset{K}\) is a subset of the dual space, as defined in Banach space theory.

Moreover, by examining the proof of \Href{Proposition}{prp:john_inclusion} (which is even more transparent in more general settings of criteria for the maximal volume position of one convex body inside another \cite[Theorem 3.8]{GLMP04}), one sees that, instead of inclusion \eqref{eq:johninclusion}, one obtains the equivalent inclusion \eqref{eq:johninclusion_with_polars}.

We now turn to an explanation on the functional side of the story.

\subsection{John's inclusion for functions and Duality}
The main problem is that a direct “translation” of John's inclusion \eqref{eq:johninclusion} to the functional setting makes no sense. 
Indeed, it is easy to see that the class of functions $w$ for which the
family $\braces{\,g \in \funpos{w} : g \leq f}$ is nonempty for any proper
log-concave function $f$ must consist of functions with bounded support. 
Hence, all considered functional analogues of the unit ball are of bounded support. 
But then, for a strictly positive log-concave function $f$ (for instance, the standard Gaussian density), there is no position of $w$ above $f.$ Thus, to obtain an analogue of \Href{Proposition}{prp:john_inclusion}, one must either relax the inclusion (e.g., by cutting off the “tails” of the functions) or, equivalently, re-formulate it in a way that can be translated into the functional setting. We adopt the latter approach via polar duality.

\subsection{John functions}

Theorem 4.1 of \cite{ivanov2022functional} essentially shows that a John function of any proper log-concave function $f$ with respect to $\hballf$ exists and is unique.
Accordingly, in this setting we refer to this unique solution as the \emph{John function} of $f$.

Furthermore, in \cite[Theorem 5.1]{ivanov2022functional} it is shown that the following John-type characterization holds:

\begin{dfn}[John's decomposition of the identity for functions]
We say that points 
${u}_1,\ldots,{u}_m \in \ball{d}\subset \Red$ and 
positive weights $c_1, \ldots ,c_m$ form  \emph{John's decomposition of the identity for functions} if they satisfy the identities:
\begin{enumerate}
\item $\sum\limits_{i \in [m]} c_i\,{u}_i\otimes {u}_i = \id_d;$
\item $\sum\limits_{i \in [m]} c_i\,\hballf\parenth{{u}_i} \cdot \hballf\parenth{{u}_i} = 1;$
\item $\sum\limits_{i \in [m]} c_i\,u_i =0.$
\end{enumerate}
\end{dfn} 

\begin{prp}\label{prp:johncondfunc} 
Let $f$ be a proper log-concave function on $\Red$ such that $\hballf \leq f.$
Then the following assertions are equivalent:
\begin{enumerate}
 \item $\hballf$ is the John function of $f.$
 \item There exist points ${u}_1,\ldots,{u}_m \in \ball{d}\subset \Red$ and positive weights
$c_1,\ldots,c_m$ forming  John's decomposition of the identity for functions, such that 
${u}_1,\ldots,{u}_m$ are contact points; that is, for each $i \in [m]$, either
$f(u_i)= \hballf(u_i)$ or $u_i$ is a unit vector on the boundary of $\supp f.$
\end{enumerate}
\end{prp}

We also bounded in \cite[Lemma 4.5]{ivanov2022functional} the supremum norm of the John function of $f$:

\begin{prp}\label{prp:norm_in_john_position}
Let $f$ be a proper  log-concave function on $\Red$ such that 
$\hballf$ is its John function. 
Then 
\[
\norm{f}_\infty \;\leq\; e^d.
\]
\end{prp} 

Interestingly, as we show in \Href{Subsection}{subsec:height_ball_john}, the bound $e^{d}$
cannot be attained for the John function, yet it is optimal in the case where $w$ in  Functional John problem \eqref{eq:john_func_problem} is the
indicator function of the unit ball (see \cite{alonso2018john}).

In view of a more general result \cite{ivanovIMRN}, we state a broader claim, whose proof we will sketch in \Href{Appendix}{sec:bound_on_height} because it closely follows the arguments from \cite[Theorem 1.1]{alonso2018john}:

\begin{lem}\label{lem:height_john_f_general}
Let $f, w \colon \Red\to[0,\infty)$ be two proper log-concave functions such that 
the support of $w$ is bounded. 
Then there is a solution $\tilde{g}$ to Functional John problem \eqref{eq:john_func_problem} and $\tilde{g}$ satisfies
\[
\norm{\tilde{g}}_{\infty} \;\leq\; \norm{f}_{\infty} 
\;\;\leq\; e^d \,\norm{\tilde{g}}_{\infty}.
\]
\end{lem}

\section{Inequalities for ``supporting'' conditions} 
The main observation in the proof of \Href{Proposition}{prp:john_inclusion} is the following “supporting” inclusion:

\emph{Assume $\ball{d}$ is a subset of a convex set $K$ in $\R^d,$
and let the unit vector $u$ lie on the boundary of $K.$ Then 
$K$ is contained in the half-space
}
\[
H_u^{\leq} 
= 
\braces{x \in \R^d \st \iprod{u}{x} \leq 1 }.
\]
 
In this section, we discuss an extension of this result to log-concave functions and derive some basic corollaries.

The following proposition follows immediately from the supporting condition for the corresponding convex functions and was formally proven in \cite[Lemma 3.1]{ivanov2022functional}:
\begin{prp}\label{prp:touching_cond_log_concave}
Let $\psi_1$ and $\psi_2$ be convex functions on $\Red$, and set $f_1= e^{-\psi_1}$ 
and $f_2 = e^{-\psi_2}.$ 
Suppose $f_2\leq f_1$ and $f_1(x_0) = f_2(x_0) >0$ at some point $x_0$ 
in the interior of the domain of $\psi_2.$ 
Assume that $\psi_2$ is differentiable at $x_0.$ Then $f_1$ and $f_2$ 
are differentiable at $x_0,$ 
$\nabla f_1(x_0)= \nabla f_2(x_0),$
and
\[
f_1(x) \;\leq\; f_2(x_0)\,\exp
\parenth{- \iprod{\nabla\psi_2(x_0)}{x - x_0}}
\]
for all $x \in \R^d.$
\end{prp}

For each $u \in \ball{d} \subset \R^d,$ define a function $\ell_{u} \colon \R^d \to [0, +\infty]$ by
\[
\ell_{u}(x) 
= 
\hballf(u)\, \exp \!\parenth{- \frac{1}{\hballf (u)^2}\,\iprod{u}{\,x - u}}
\]
if $\enorm{u} < 1;$ and by 
\[
\ell_{u}(x) =
\begin{cases} 
  0,      & \iprod{x}{u} \,\geq\, 1, \\ 
  +\infty,& \iprod{x}{u} \,<\, 1,
\end{cases}
\]
if $\enorm{u} = 1.$

As a corollary of \Href{Proposition}{prp:touching_cond_log_concave}, we get:
\begin{cor}\label{cor:touchingballbound}
 Let $g$ be a log-concave function on $\R^d$ such that $g \geq \hballf,$ and let $u$ be a  vector from $  \ball{d}$ with $g(u) = \hballf(u).$ Then  $g \leq \ell_u.$
\end{cor}

This observation is the key technical tool in our proof of \Href{Theorem}{thm:john_inclusion_one_f}.

\subsection{The origin yields almost everything}
We will often use the following direct corollary of the definition of the polar function: 
\begin{claim}\label{claim:norm_via_polar_function}
Assume a bounded, non-negative function $g \colon \R^d \to [0, \infty)$ takes at least one positive value. Then
\[
\polarset{g}(0) \;=\; \frac{1}{\norm{g}_\infty}.
\]
\end{claim}

Define $\zeta\colon [0,1] \to [0, \infty)$ by 
\(\zeta(t) = t^{-t}\) on \((0,1]\) and \(\zeta(0) = 1.\)
\begin{claim}
\label{claim:technical_bump_func}
The function $\zeta$ is continuous and log-concave on $[0,1]$, and it attains its minimum, equal to 1, only at $0$ and $1$.
\end{claim}
\begin{proof}
By routine calculus, the second derivative of $t \ln t$ on $(0,1]$ is $\frac{1}{t}$, 
so $\zeta$ is log-concave on $(0,1]$. Clearly, $\zeta(t) > 1$ for all $t \in (0,1]$.
Moreover, $\zeta(t)$ converges monotonically to 1 as $t\to 0.$
\end{proof}

For a subset $S \subset \R^d,$ we denote by $\chi_S$ the indicator function
of $S,$ that is,
\[
\chi_S(x) 
=  
\begin{cases}
1,& x \in S,\\
0,& x \notin S.
\end{cases}
\] 
\begin{claim}
\label{claim:polar_of_log-linear}
For a vector $u \in \ball{d} \subset \R^d$ with $\enorm{u} < 1,$ we have 
\[
\polarset{\parenth{\ell_u}} 
= 
\frac{1}{\hballf(u)}\,
\exp\!\parenth{- \frac{\enorm{u}^2}{\hballf (u)^2}}
\,\chi_{\frac{u}{\hballf (u)^2}}.
\]
\end{claim}
\begin{proof}
If $u = 0,$ then $\polarset{\parenth{\ell_u}} = \chi_{\,u}.$ 
If $u \neq 0,$ it follows from the definition of the polar function that
\[
\polarset{\parenth{\ell_u}}(p) 
=  
\frac{1}{\hballf(u)}\,\exp\!\parenth{- \frac{\enorm{u}^2}{\hballf (u)^2}} \cdot
\inf_{x \in \R^d}  
\exp\! \parenth{- \iprod{x}{\,p - \frac{u}{\hballf (u)^2}}}.
\] 
Clearly, the above infimum is zero unless 
\(p = \frac{u}{\hballf (u)^2},\)
in which case it equals 1. This completes the proof of \Href{Claim}{claim:polar_of_log-linear}.
\end{proof}

\begin{lem}
\label{lem:polar_func_lower_bound}
Assume a bounded, non-negative function $g \colon \R^d \to [0, \infty)$ satisfies $\norm{g}_\infty \geq 1$ and ${g} \leq \ell_u$ for some  
$u \in \ball{d}$ with $0 < \enorm{u} < 1.$ 
Then  
\begin{equation}
\label{eq:bound_on_polar_contact_point}
\polarset{g}(u) \;\geq\; \frac{\polarset{g}(0)}{e}.
\end{equation}
\end{lem}
\begin{proof}
By \Href{Claim}{claim:norm_via_polar_function} and since $g$ is bounded,
\(
0 < \polarset{g}(0) \leq 1.
\)
Since ${g} \leq \ell_u,$ we have $\polarset{g} \geq \polarset{\parenth{\ell_u}}.$ By log-concavity,
\[
\polarset{g}\!\!\parenth{\,t\,\frac{u}{\enorm{u}}}
\;\geq\;
\parenth{\polarset{\parenth{\ell_u}}\!\parenth{\frac{u}{\hballf (u)^2}}}^{\frac{\hballf (u)^2}{\enorm{u}}\,t}
\,
\parenth{\polarset{g}(0)}^{\,1 - \frac{\hballf (u)^2}{\enorm{u}}\,t}
\]
for all $t \in \brackets{0, \frac{\enorm{u}}{\hballf (u)^2}}.$
By \Href{Claim}{claim:polar_of_log-linear},
\[
\polarset{g}\!\!\parenth{\,t\,\frac{u}{\enorm{u}}} 
\;\geq\; 
\parenth{\hballf(u)}^{-\frac{\hballf (u)^2}{\enorm{u}}\,t}
\cdot 
\parenth{\,\polarset{g}(0)}^{\,1 - \frac{\hballf (u)^2}{\enorm{u}}\,t}
\cdot 
e^{- \enorm{u}\,t}
\]
for all $t \in \brackets{0, \frac{\enorm{u}}{\hballf (u)^2}}.$
Choosing $t = \enorm{u}$ (note $\enorm{u} \leq \frac{\enorm{u}}{\hballf (u)^2})$, we obtain
\[
\polarset{g}(u)  
\;\geq\; 
\parenth{\hballf(u)}^{-\hballf (u)^2}
\,\parenth{\,\polarset{g}(0)}^{\,1 - \hballf (u)^2} 
\cdot 
e^{- \enorm{u}^2}.
\]
Next, by \Href{Claim}{claim:technical_bump_func}, 
\[
\parenth{\hballf(u)}^{-\hballf (u)^2}
\;=\;
\sqrt{\,\parenth{\hballf (u)^2}^{-\hballf (u)^2}} 
\;>\; 1.
\]
Combining this with $\enorm{u} < 1,$ we get
\[
\polarset{g}(u)  
\;\geq\;
\parenth{\,\polarset{g}(0)}^{\,1 - \hballf (u)^2}
\cdot 
e^{-1}.
\]
By \Href{Claim}{claim:norm_via_polar_function}, $\polarset{g}(0) \leq 1,$ so
\[
\polarset{g}(0)^{\,1 - \hballf (u)^2}
\;\geq\;
\polarset{g}(0).
\]
Hence \eqref{eq:bound_on_polar_contact_point} follows.
\end{proof}

\section{Properties of decompositions of the identity}
In this section, we derive several properties of John's decompositions of the identity that are needed for the proof of \Href{Theorem}{thm:john_inclusion_one_f}.  
We begin by establishing a relaxed version of the classical John's inclusion \eqref{eq:johninclusion}.  
Then we show that it suffices to consider only everywhere positive functions, thereby allowing us to disregard the case of unit vectors in John's decompositions for functions.  
This is a purely technical observation that simplifies the subsequent arguments.  
Recall that in the case $\enorm{u}=1$, the corresponding function $\ell_u$ is unbounded on a half-space.

\begin{lem}\label{lem:John_inclusion_relaxed_case}
Let vectors ${u}_1,\ldots,{u}_m \in \ball{d}\subset \Red$ and 
positive weights $c_1,\ldots,c_m$ satisfy
\[
\sum\limits_{i \in [m]} c_i {u}_i\otimes {u}_i = \id_d
\quad \text{and} \quad \sum\limits_{i \in [m]} c_i u_i =0.
\]
Denote by $K$ the convex hull of ${u}_1,\ldots,{u}_m$. 
Then 
\[
{\ball{d}} \subset   \sum\limits_{i \in [m]} c_i \cdot K.
\]
\end{lem}

\begin{proof}
Since ${u}_1,\ldots,{u}_m \in \ball{d},$ we know
\begin{equation}
\label{eq:glmp_inclusion}
\enorm{x} \geq \max\limits_{i \in [m]} \iprod{x}{u_i}
\quad\text{for all } x \in \Red.
\end{equation}
Hence, 
\[
-x  = - \id_d (x) 
= -\sum\limits_{i \in [m]} c_i \iprod{u_i}{x} u_i 
\stackrel{(\ast)}{=}
\enorm{x} \sum\limits_{i \in [m]} c_i u_i 
-\sum\limits_{i \in [m]} c_i \iprod{u_i}{x} u_i 
= 
\sum\limits_{i \in [m]} c_i \parenth{\enorm{x} - \iprod{u_i}{x}} u_i,
\]
where in $(\ast)$ we used $\sum\limits_{i \in [m]} c_i u_i = 0.$

By \eqref{eq:glmp_inclusion}, $\enorm{x} - \iprod{u_i}{x}$ is nonnegative for every $i \in [m]$. Thus,
\[
-x 
\;\in\; \sum\limits_{i \in [m]}  
c_i\,\parenth{\enorm{x} - \iprod{u_i}{x}}\,K 
\;=\;
\parenth{\sum\limits_{i \in [m]} c_i}\,\enorm{x}\,K 
\;-\; 
\parenth{\iprod{\sum\limits_{i \in [m]} c_i u_i}{x}}\,K.
\]
Using the identities $\sum\limits_{i \in [m]} c_i u_i=0$ and
denoting $C_{\text{tr}} = \sum\limits_{i \in [m]} c_i,$ 
we get 
\[
-x  \in C_{\text{tr}}\;\enorm{x}\;K.
\]
In particular, if $\enorm{x} \leq 1$ (that is, $-x \in \ball{d}$), it follows that $-x \in C_{\text{tr}} K.$ The lemma now follows.
\end{proof}

\begin{cor}\label{cor:john_inclusion_dec_for_f}
Assume vectors 
${u}_1,\ldots,{u}_m \in \ball{d}\subset \Red$ and 
positive weights $c_1, \ldots ,c_m$ \emph{form  John's decomposition of the identity for functions}.  
Then $\sum\limits_{i \in [m]} c_i = d+1$ and the convex hull $K$ of ${u}_1,\ldots,{u}_m$ contains the ball $\frac{\ball{d}}{d+1}.$
\end{cor}

\begin{proof}
Taking traces in the first equation from the definition of  John's decomposition of the identity for functions and adding the second, 
we conclude $\sum\limits_{i \in [m]} c_i = d+1.$
Then, by \Href{Lemma}{lem:John_inclusion_relaxed_case}, $K$ contains the ball $\frac{\ball{d}}{d+1}.$
\end{proof}

\subsection{Reduction to everywhere positive functions}
There are some technical difficulties in the case of ``contact'' at the boundary of the unit ball. There are several ways to circumvent these; we choose to employ a certain limit argument.

\begin{dfn} 
Assume vectors 
${u}_1,\ldots,{u}_m \in \ball{d}\subset \Red$ and 
positive weights $c_1, \ldots ,c_m$ form  {John's decomposition of the identity for functions}. 
We call a function ${f}$ of the form 
\[
{f} = \min\limits_{i \in [m]} \ell_{u_i}
\]
a \emph{John bump function}. If all points ${u}_1,\ldots,{u}_m$ additionally lie in the interior of the unit ball, we call 
${f}$ a \emph{regular John bump function}.
\end{dfn}

By \Href{Proposition}{prp:johncondfunc}, $\hballf$ is the John function of any John bump function.

\begin{lem}\label{lem:reduction_to_positive_f}
Fix a point $p \in \R^d$ and a non-negative constant $\alpha.$
The following assertions are equivalent:
\begin{enumerate}
\item\label{ass:reduction_to_positive_f_1} For every proper log-concave function $f$ on $\R^d$ with $\hballf$ as its John function, the inequality
\[
 \alpha \;\leq\; \polarset{f}(p)
\] 
holds.
\item\label{ass:reduction_to_positive_f_2} The same inequality 
\[
 \alpha \;\leq\; \polarset{f}(p)
\] 
holds for every \emph{regular} John bump function $f$.
\end{enumerate}
\end{lem}

\begin{proof}
If $\alpha = 0$, there is nothing to prove, so assume $\alpha>0$.

By \Href{Proposition}{prp:johncondfunc} and \Href{Corollary}{cor:touchingballbound}, any proper log-concave function $f$ on $\R^d$ for which $\hballf$ is the John function lies below some John bump function $\tilde{f}.$ Consequently, $\tilde{f}$ is a proper log-concave function satisfying $\polarset{f} \geq \polarset{\tilde{f}}$. 
Hence, \eqref{ass:reduction_to_positive_f_1} is equivalent to:
\\ 
\noindent
\emph{``The inequality $\alpha \leq \polarset{f}(p)$ holds for every John bump function $f$.''}

This in turn clearly implies \eqref{ass:reduction_to_positive_f_2}. 
Thus, it suffices to prove \eqref{ass:reduction_to_positive_f_2} $\Rightarrow$ \eqref{ass:reduction_to_positive_f_1}.

By standard convex analysis (cf.\ \cite[Chapter 7]{rockafellar2009variational}), it is enough to construct a sequence $\{f_n\}$ of \emph{regular} John bump functions hypo-convergent to a given John bump function $f$, i.e.,
\[
\limsup_{n \to \infty} f_n(x_n) \leq f(x) 
\quad\text{for every } x_n \to x,
\]
and
\[
\liminf_{n \to \infty} f_n(x_n) \geq f(x) 
\quad\text{for some } x_n \to x.
\]

\noindent
Let us construct such a sequence.
Take vectors $u_1,\ldots,u_m$ and weights $c_1,\ldots,c_m$ forming  John's decomposition of the identity for functions, so that 
\[
f = \min\limits_{i \in [m]} \ell_{u_i}.
\]
For these vectors $u_1, \dots, u_m$ and weights $c_1, \dots, c_m,$ define a set of vectors $\bar{U} \subset \R^{d+1}$ and an associated multi-set of weights $\tilde{C}$ as follows:
\begin{enumerate}
\item If $\enorm{u_i} < 1,$ add the vectors
$(u_i,  \pm \hballf( u_i)) \in \R^{d+1}$ to $\bar{U}$, each with weight
$\frac{c_i}{2};$
\item If $\enorm{u_i} = 1,$ add the vector
$(u_i,  0) \in \R^{d+1}$ to $\bar{U}$ with weight
$c_i.$
\end{enumerate}
Since the original vectors and weights form  John's decomposition of the identity for functions in $\R^d$, the vectors in $\bar{U}$ and weights in $\tilde{C}$ satisfy the $(d+1)$-dimensional version of \eqref{eq:johncond}, namely
\[
\sum \tilde{c}\,{\bar{u} \otimes \bar{u}}= \id_{d+1}
\quad \text{and} \quad 
\sum \tilde{c}\,\bar{u}  = 0,
\]
where the sums run over $\bar{u} \in \bar{U}$ and associated weights $\tilde{c} \in \tilde{C}.$

A key observation is that these equations are invariant under orthogonal transformations. 
Let $H_{d-1}$ be a linear hyperplane in $\R^d$ avoiding $\bar{U}\cap \R^d.$ 
Denote by $P_d$ the orthogonal projection of $\R^{d+1}$ onto $\R^d,$
and let $O_n$ denote the rotation around $H_{d-1}$ by angle $\frac{1}{n}$ in a fixed direction. 
Define
\[
U_n = P_d \circ O_n \parenth{\bar{U}}.
\]
For each natural $n$, the vectors in $U_n$ together with the same weights $\tilde{C}$ form  John's decomposition of the identity for functions in $\R^d$. 

For sufficiently large $n$, all vectors in $U_n$ lie strictly inside the unit ball $\ball{d}$. Consequently, the functions
\[
f_n = \min\limits_{u \in U_n} \ell_{u}
\]
are \emph{regular} John bump functions. By a standard limit argument, they hypo-converge to $f$. This completes the proof.
\end{proof}
\section{Inequalities for John's function}

\subsection{John's inclusion for log-concave functions} 

\Href{Theorem}{thm:john_inclusion_one_f} is a direct consequence of the following:

\begin{thm}
\label{thm:John_inclusion_hballf}
Let $f$ be a proper  log-concave function on $\Red$ such that 
$\hballf$ is its John function. Then
\[
   e^{-(d+1)} \cdot \chi_{\frac{\ball{d}}{d+1}} \;\leq\; \polarset{f}.
\]
\end{thm}

\begin{proof}
By \Href{Lemma}{lem:reduction_to_positive_f}, it suffices to consider the case of a {regular} John bump function.  Thus, assume 
\[
f = \min\limits_{i \in [m]} \ell_{u_i},
\]
where each $u_i$ lies in the interior of $\ball{d}$ and $c_1,\ldots,c_m$ are the associated weights from  John's decomposition of the identity for functions. 

Since $\hballf \le f$, \Href{Proposition}{prp:norm_in_john_position} and \Href{Claim}{claim:norm_via_polar_function} imply
\[
e^{-d} \;\leq\; \polarset{f}(0) \;\leq\; 1.
\]
Using \eqref{eq:bound_on_polar_contact_point} of \Href{Lemma}{lem:polar_func_lower_bound}, 
\[
\polarset{f}(u_i) \;\geq\; \frac{\polarset{f}(0)}{e} \;\geq\; e^{-(d+1)}.
\]
By log-concavity, $\polarset{f}$ is at least $e^{-(d+1)}$ throughout the convex hull $K$ of the $u_i$'s.  By 
\Href{Corollary}{cor:john_inclusion_dec_for_f}, $K$ contains $\frac{\ball{d}}{d+1}.$
Hence, $\polarset{f}$ remains at least $e^{-(d+1)}$ on $\frac{\ball{d}}{d+1}.$ 
This completes the proof of \Href{Theorem}{thm:John_inclusion_hballf}.
\end{proof}

\begin{rem}
We note that the original John's inclusion~\eqref{eq:johninclusion} can be derived using our approach.
Indeed, if the unit ball $\ball{d}$ is the John ellipsoid of a convex body $K$, then
$\hballf$ is the John function of the indicator function of $K$.
Hence, the non-zero vectors from the corresponding John's decomposition of the identity for functions
satisfy~\eqref{eq:johncond}.
Therefore, the quantity $C_{\text{tr}}$ in \Href{Lemma}{lem:John_inclusion_relaxed_case} is equal to $d$,
and the desired inclusion follows.
\end{rem}

\Href{Lemma}{lem:flat_functions_inclusion} follows from \Href{Theorem}{thm:John_inclusion_hballf}:
\begin{proof}[Proof of \Href{Lemma}{lem:flat_functions_inclusion}]
Without loss of generality, place $f$ so that 
$\hballf$ is its John function.  
By \Href{Theorem}{thm:John_inclusion_hballf}, we get
\[
 \hballf \;\leq\; f  \;\leq\; e^{-\frac{\enorm{x}}{d+1} + (d+1)}.
\]
Also,  \(\frac{1}{\sqrt{d+1}}\,\chi_{\sqrt{\frac{d}{d+1}}\,\ball{d}} \;\leq\; \hballf.\) 
Combining these two inequalities, we see that 
\[
\tilde{f} 
= 
\sqrt{d+1}\,\cdot\, f \circ \brackets{\sqrt{\frac{d}{d+1}}\,\id_d}
\]
satisfies 
\[
\chi_{\ball{d}} 
\;\leq\; 
\tilde{f} 
\;\leq\; 
\sqrt{d+1}\,\cdot\,e^{-\sqrt{\frac{d}{d+1}}\;\frac{\enorm{x}}{d+1} + (d+1)}.
\]
The lemma follows from the elementary inequality 
\[
\sqrt{\frac{d}{d+1}}\;\frac{1}{d+1}  \;>\; \frac{1}{d+2},
\]
valid for any natural \(d\).
\end{proof}

\subsection{What is the optimal bound on the height?}
\label{subsec:height_ball_john}
\begin{lem}\label{lem:john_height_not_attained}
There is a positive constant \(\epsilon_d\) such that the following holds:
If \(f\) is a proper log-concave function on \(\Red\) with \(\hballf\) as its John function, then 
\[
\norm{f}_{\infty} \;\leq\; e^{d} - \epsilon_d.
\]
\end{lem}

\begin{proof}
By \Href{Claim}{claim:norm_via_polar_function} and by \Href{Lemma}{lem:reduction_to_positive_f}, we may restrict to the case of a \emph{regular} John bump function.  Hence, suppose
\[
f = \min\limits_{i \in [m]} \ell_{u_i},
\]
where each \(u_i\) lies strictly inside \(\ball{d}\), and \(c_1,\ldots,c_m\) are the corresponding weights from  John's decomposition of the identity for functions. Denote the convex hull of these \(u_i\)'s by \(K\).

We split the proof into two cases.

\noindent
\textit{Case 1: A bound when a contact point is near the ``North pole.''}

We claim that there exists a positive constant $\gamma_d$ with the following property: if $\hballf(u_i) \ge 1 - \gamma_d$ for some $i \in [m]$, then
$\norm{f}_{\infty} \le e^d - 1$.
Indeed, by monotonicity arguments, \(\max f\) is achieved in \(\polarset{K}.\) By 
\Href{Corollary}{cor:john_inclusion_dec_for_f}, 
\(\polarset{K} \subset (d+1)\,\ball{d}.\) The existence of such \(\gamma_d\) follows from continuity.

\medskip
\noindent
\textit{Case 2: A lower bound on \(\polarset{f}(0)\) otherwise.}

Next, assume \(\hballf(u_i) < 1 - \gamma_d\) for all \(i \in [m].\)
By \Href{Claim}{claim:norm_via_polar_function}, it suffices to bound \(\polarset{f}(0)\) from below. 
For every \(i \in [m],\)
\[
\polarset{f}\!\parenth{\frac{u_i}{\hballf(u_i)^2}} 
\;\geq\; 
\polarset{\parenth{\ell_{u_i}}}\!\parenth{\frac{u_i}{\hballf(u_i)^2}} 
\;=\;  
\frac{1}{\hballf(u_i)}\, e^{- \frac{\enorm{u_i}^2}{\hballf(u_i)^2}}
\]
by \Href{Claim}{claim:polar_of_log-linear}.
Using first \(\sum\limits_{i=1}^m c_i\,\hballf(u_i)^2 = 1\) and then \(\sum\limits_{i=1}^m c_i\,u_i=0,\) we get
\[
\prod\limits_{i=1}^m 
\parenth{\polarset{f}\!\parenth{\frac{u_i}{\hballf(u_i)^2}}}^{c_i\,\hballf(u_i)^2} 
\;\leq\; 
\polarset{f}\!\parenth{\sum\limits_{i=1}^m c_i\,u_i} 
\;=\; 
\polarset{f}(0)
\]
by the log-concavity of \(\polarset{f}\). Thus,
\[
\polarset{f}(0) 
\;\geq\; 
\prod\limits_{i=1}^m  
\parenth{\frac{1}{\hballf(u_i)}\, e^{- \frac{\enorm{u_i}^2}{\hballf(u_i)^2}}}^{\,c_i\,\hballf(u_i)^2} 
\;=\; 
e^{-\sum\limits_{i \in [m]} c_i \enorm{u_i}^2} 
\prod\limits_{i=1}^m  
{\hballf(u_i)}^{-\,c_i\,\hballf(u_i)^2}
\;=\; 
e^{-d} 
\prod\limits_{i=1}^m  
{\hballf(u_i)}^{-\,c_i\,\hballf(u_i)^2}.
\]
Each factor \({\hballf(u_i)}^{-\,c_i\,\hballf(u_i)^2} \) is at least one. It remains to show that at least one of these factors is strictly greater than \(1 + \delta_d\) for some \(\delta_d>0.\) By Carath\'eodory's theorem \cite{caratheodory1911variabilitatsbereich}, we can assume \(m \le 4d^2\). Hence there exists some \(j \in [m]\) for which 
\[
c_j\,\hballf^2(u_j) 
\;\ge\; 
\frac{1}{m} 
\;\ge\; 
\frac{1}{4d^2}.
\]
Then 
\[
\hballf(u_j) 
\;\geq\; 
\frac{1}{2d}\,\sqrt{\frac{1}{c_j}}
\;\;\geq\;\;
\frac{1}{2d}\,\sqrt{\frac{1}{\sum_{i \in [m]} c_i}}
\;\;\geq\;\;
\frac{1}{4d^2}.
\]
This ensures the product above exceeds \(e^{-d}\) by some fixed gap \(\delta_d>0,\) 
so that ultimately
\[
\polarset{f}(0) 
\;>\; e^{-d}\,(1 + \delta_d),
\]
hence 
\(\norm{f}_{\infty} = \frac{1}{\polarset{f}(0)} < e^{d} - \epsilon_d\)
for some \(\epsilon_d>0\) depending only on \(d\). 
\end{proof}

\begin{rem}
It is not difficult to show that \(c_i \le 2\) for all \(i \in [m]\).
\end{rem}

\section{Absence of L\"owner's inclusion or misbehavior of tails}
\label{sec:absence_lowner_inclusion}

We refer the interested reader to \cite{ivanovIMRN} for a detailed discussion on L\"owner functions and their relation to John functions.
Below we recall several necessary definitions.

We will call a \emph{L\"owner function of $f$ with  respect to a given function $w$} any solution to the  

\medskip
\noindent
\textbf{Functional L\"owner problem:} Find
\begin{equation}\label{eq:lowner_func_problem}
\min\limits_{g \in \funpos{w} } 
	\int_{\Red} g
	\quad \text{subject to} \quad
	 f \leq g.
\end{equation}

As in the case of  Functional John problem \eqref{eq:john_func_problem}, the set 
\(\braces{g \in \funpos{w} \st f \le g}\) can be empty.
However, it is not hard to describe the set of functions $w$ for which it is nonempty for any proper log-concave $f$ --- specifically, those whose polar functions have bounded support. In terms of the original function $w,$ this property characterizes the behavior of the ``tails'' of $w$ at infinity. We provide the following equivalent description  without proving the equivalence here: there is a position $g$ of $w$ such that 
\[
e^{-\enorm{x}} \;\leq\; g.
\]

Our goal in this section is to show that for a reasonably large class of functions $w$, there is no inclusion of the form 
\[
f \;\leq\; L
\quad \text{and} \quad
\polarset{f} \;\leq\; \alpha \cdot L \circ  \frac{\id_d}{\alpha},
\]
where $L$ is a solution to  Functional L\"owner problem \eqref{eq:lowner_func_problem}.

The idea is that the ``tails'' of a function $f$ impose multiple restrictions on the set of positions
\(\braces{g \in \funpos{w} \st f \le g}\).  

\begin{lem} \label{lem:no_lowner_inclusion}
Let $L$ be one of the functions $e^{-\enorm{x}^p}$ with $p \geq 1$, or $\polarset{\parenth{\hballf^s}}$ with $s > 0.$ Define
\[
L_{+}(x) = \begin{cases}
L(x), & \text{if } \iprod{x}{e_1} \geq 0, \\
0,    & \text{if } \iprod{x}{e_1} < 0.
\end{cases}
\]
Then $L$ is the unique L\"owner function of $L_{+}$ with respect to $L.$
Moreover, the set of positions
\(\braces{g \in \funpos{L} \st \polarset{\parenth{L_{+}}} \leq g}\)
is empty.
\end{lem}

\begin{proof}
The emptiness of 
\(\braces{g \in \funpos{L} \st \polarset{\parenth{L_{+}}} \leq g}\)
follows immediately from the observation that
\(\polarset{\parenth{L_{+}}}\) is not proper: indeed,
\(\polarset{\parenth{L_{+}}}(-t e_1) = 1\) for all \(t \geq 1.\)

Now, assume
\(\alpha\,L\!\parenth{A(x-a)} \ge L_{+}(x)\) for all \(x \in \R^d.\)
Clearly, \(\alpha \geq 1\); and if \(\alpha = 1,\) then \(a = 0.\)

Fix any unit vector \(u\) with \(\iprod{u}{e_1} \ge 0,\) and denote \(v = Au.\)
For all \(t>0,\)
\[
\alpha 
\;\ge\; 
\frac{L_{+}(tu)}{L\!\parenth{t\,v - Aa}} 
\;=\; 
\frac{L\!\parenth{tu}}{L\!\parenth{t\,v - Aa}},
\]
where in the last step we used \(\iprod{u}{e_1}\ge 0,\) so \(L_{+}(tu) = L\parenth{tu}\). 

Taking the limit as \(t\to\infty\), and noting that $L$ is rotationally invariant, we conclude \(\enorm{u} \ge \enorm{v}.\) Hence, \(A\ball{d}\subseteq \ball{d}\). Consequently, the integral of \(\alpha\,L\!\parenth{A(x-a)}\) is at least that of \(L\). Moreover, equality is attained if \(\alpha =1\), \(a =0\), and \(A\) is an orthogonal transformation. The lemma follows.
\end{proof}

We note that for the cases \( L = e^{-\enorm{x}}\) and 
\(L = \polarset{\parenth{\hballf^s}}\) with \(s > 0,\) similar arguments can be derived from the L\"owner condition in \cite{ivanovIMRN}.  It may be surprising to find such a simple example for the Gaussian case $L(x) = e^{-\enorm{x}^2}.$

\section{Discussions}

The primary purpose of this paper was to demonstrate the possibility of extending the John inclusion to the functional setting. However, the results raise several natural questions:
\begin{enumerate}
\item The set of possible weights \(c_1, \dots, c_m\) such that there exist unit vectors \(u_1, \dots, u_m\) satisfying 
\(\sum_{i \in [m]} c_i\,u_i \otimes u_i = \id_d\)
is a convex polytope \cite[Lemma 2.4]{ivanov2020volume}. 
What is the set of possible weights appearing in  John's decomposition of the identity for functions? 

For instance, in the classical setting a weight cannot exceed 1, but in the functional setting it can (yet, in our context, it cannot exceed 2).

\item In \cite{ivanov2022functional}, the authors considered the solution to the Functional John problem \eqref{eq:john_func_problem} with \(w = \hballf^s\) for \(s> 0.\) We claim that \Href{Lemma}{lem:john_height_not_attained} can be generalized to this case directly, but our approach to \Href{Theorem}{thm:john_inclusion_one_f} provides a reasonable bound only if \(s \geq 1.\) It remains open what happens in the regime \(s \in (0,1]\), especially in the limit \(s \to 0.\) For example, is there a John-type inclusion for the solution of the Functional John problem \eqref{eq:john_func_problem} with \(w = \chi_{\ball{d}}\)?
Recall that \(w = \chi_{\ball{d}}\,\) was a starting point of the entire topic in \cite{alonso2018john}.

\item The polar function of a Gaussian density is again a Gaussian density. It is highly intriguing to investigate whether the duality between John and L\"owner ellipsoids, discussed in the Introduction, can be generalized to the functional setting for \(w = e^{-\enorm{x}^2}\) in problems \eqref{eq:john_func_problem} and \eqref{eq:lowner_func_problem}. 
\end{enumerate}

\appendix
\section{Bound on the ``height''}
\label{sec:bound_on_height}

The idea behind the proof of \Href{Lemma}{lem:height_john_f_general} is to construct an ``extremal curve'' of positions, starting with a maximal one, and to use Minkowski's determinant inequality 
\begin{equation}
\label{eq:minkowski_det_ineq}
\parenth{\det\parenth{\lambda A + \parenth{1 - \lambda}B}}^{1/d} \;\ge\;
\lambda\,\parenth{\det A}^{1/d} \;+\;
\parenth{1 - \lambda}\,\parenth{\det B}^{1/d},
\end{equation}
to obtain certain convexity properties of the integrals of these positions along the curve. We need to use positive definite matrices to apply Minkowski's determinant inequality. Let us introduce several definitions:

We denote by
\[
\funppos{g} \;=\; \braces{\alpha\,g\parenth{A\,x + a}
\st 
A\in\R^{d\times d}\text{ is positive definite},\;\alpha>0,\;a\in\R^d}
\]
the \emph{positive positions} of \(g\).

\medskip

\noindent
\textbf{Fixed-height John problem for \(f\) and \(w\):}
Find
\begin{equation}\label{eq:john_problem_pos_height}
\max\limits_{g \in \funppos{w}} 
	\int_{\R^d} g
	\quad 
	\text{subject to} 
	\quad
	g \;\leq\; f  
	\quad
	\text{and}
	\quad
	\norm{g}_{\infty} = \xi.
\end{equation}

Log-concavity allows us to consider a certain average position:

\begin{prp}[Inner interpolation of functions]\label{prp:inner-function-interpolation}
Let $f \colon \R^d \to [0,+\infty)$ be a log-concave function and  $g \colon \R^d \to [0,+\infty)$ be any function.  
Let $\alpha_1,\alpha_2>0$, $A_1, A_2$ be non-singular $d\times d$ matrices,
and $a_1,a_2 \in \R^d$ satisfy
\[
\alpha_1\,g\!\parenth{A_1^{-1}\parenth{x - a_1}} \;\leq\; f(x)
\quad
\text{and}
\quad
\alpha_2\,g\!\parenth{A_2^{-1}\parenth{x - a_2}} \;\leq\; f(x)
\]
for all $x\in\R^d$.  
Let $\beta_1,\beta_2>0$ be such that $\beta_1+\beta_2=1$. Define
\[
\alpha = \alpha_1^{\beta_1}\,\alpha_2^{\beta_2},
\quad
A = \beta_1 A_1 + \beta_2 A_2,
\quad
a = \beta_1 a_1 + \beta_2 a_2.
\]
Assume $A$ is non-singular. Then
\[
\alpha\,g\!\parenth{A^{-1}\parenth{x - a}} \;\leq\; f(x).
\]
If $A_1$ and $A_2$ are positive definite, and $g$ is integrable, then
\[
\int_{\R^d} \alpha\,g\!\parenth{A^{-1}\parenth{x - a}}\,dx
\;\geq\;
\parenth{\int_{\R^d}\alpha_1\,g\!\parenth{A_1^{-1}\parenth{x - a_1}}\,dx}^{\beta_1}
\parenth{\int_{\R^d}\alpha_2\,g\!\parenth{A_2^{-1}\parenth{x - a_2}}\,dx}^{\beta_2},
\]
with equality if and only if $A_1 = A_2.$
\end{prp}
The proposition is simple and purely technical; it was formally proven in \cite[Lemma 4.8]{ivanovIMRN}.

\begin{lem}\label{lem:uniqueness_det_fixed_height_john}
Let $f, w \colon \R^d \to [0,\infty)$ be two proper log-concave functions such that 
$w \leq f.$
Then there is a solution to Fixed-height John problem \eqref{eq:john_problem_pos_height} for $f$ and $w$  with $\xi = \norm{w}_{\infty}.$
Moreover, if $w$ has bounded support, then the solution to Fixed-height John problem \eqref{eq:john_problem_pos_height} with $\xi = \norm{w}_{\infty}
$ is unique.
\end{lem}

\begin{proof}
The lemma follows from Section 6 of \cite{ivanovIMRN}. The key point is that the set of \(\parenth{A,\alpha,a} \in \R^{d\times d}\times\R\times\R^d\) satisfying
\[
\alpha\,w\!\parenth{A^{-1}\parenth{x-a}} \;\leq\; f(x)
\]
is either compact or empty (see \cite[Lemma 6.1]{ivanovIMRN}). Existence thus follows by a standard compactness argument.

We only need to show uniqueness when $w$ has bounded support, which is achieved by a slight modification of \cite[Proposition 6.2]{ivanovIMRN}.
Let $A_1$ and $A_2$ be rank-$d$ positive definite matrices, $a_1,a_2 \in \R^d$, such that the functions
\[
g_1(x) = \alpha\,w\!\parenth{A_1^{-1}\parenth{x - a_1}}
\quad
\text{and}
\quad
g_2(x) = \alpha\,w\!\parenth{A_2^{-1}\parenth{x - a_2}}
\]
are both solutions to Fixed-height John problem \eqref{eq:john_problem_pos_height} for $f$ and $w$. In particular, their integrals are equal. 
By \Href{Proposition}{prp:inner-function-interpolation}, it follows that $A_1 = A_2.$ Hence the graphs of $g_1$ and $g_2$ differ by a translation.

Denote by $\mathrm{hypo}\,g$ the \emph{hypograph} of a nonnegative function $g$:
\[
\mathrm{hypo}\,g = \braces{(x,t) \st 0 \,\le\,t\,\le\,g(x)}.
\]
Because $f$ is log-concave, the set $\mathrm{hypo}(g_1) + [0,2v]\subset \mathrm{hypo}(f)$ for some non-zero $v \in \R^d$.  
We claim that there is a position $g$ of $w$ under $f$ such that $\int_{\R^d} g > \int_{\R^d} g_1$. 

Indeed, consider $g_{1.5}(x) = g_1\!\parenth{x - v}$. Clearly, 
$\mathrm{hypo}(g_{1.5}) \subset \mathrm{hypo}(g_1) + [0,2v] \subset \mathrm{hypo}(f).$
Let $g_{1.5}$ attain its maximum at $z$. Then $z$ belongs to all non-empty level sets 
\[
\brackets{g_{1.5}>  \Theta} = \braces{x \in \R^d \st g_{1.5}(x) > \Theta}
\]
of $g_{1.5}$ and positive $\Theta$, which are compact and convex because $g_{1.5}$ is log-concave with bounded support.  Let $S_{\epsilon}$ be the linear transformation that scales $\R^d$ in the direction of $v$ by the factor $1 + \epsilon$. Then, for sufficiently small positive $\epsilon$,  
\[
S_{\epsilon}\parenth{\brackets{g_{1.5} > \Theta} - z}  +z
\;\subset\; 
\brackets{g_1 > \Theta} + [0,2v]
\]
holds for all $\Theta \in \parenth{0,\norm{g_1}_{\infty}}.$
That is,  
\[
S_{\epsilon}\parenth{\mathrm{hypo}(g_{1.5}) - z}  +z
\;\subset\; 
\mathrm{hypo}(f).
\]
However, the left-hand set above is the hypograph of some positive position $g$ of $w$.  Uniqueness follows.
\end{proof}

By compactness and log-concavity, we have
\begin{lem}
Let $f, w \colon \R^d \to [0,\infty)$ be two proper log-concave functions such that 
the support of $w$ is bounded. Then for any $\xi \in \parenth{0, \norm{f}_{\infty}}$, there is a positive position $g$ of $w$ below $f$ such that $\norm{g}_{\infty} = \xi.$
\end{lem}

\noindent
\Href{Lemma}{lem:height_john_f_general} follows from the previous three statements and the following distilled version of \cite[Theorem 1.1]{alonso2018john}:

\begin{lem}
Let $f, w \colon \Red\to[0,\infty)$ be two proper log-concave functions such that 
$w$ is a John function for $f$ with respect to $w.$  Additionally, assume that for every $\xi \in \parenth{0, \norm{f}_\infty}$ there is a positive position $g$ of $w$ below $f$ with $\norm{g}_\infty = \xi.$ Then
\[
\norm{w}_{\infty} \;\leq\; \norm{f}_{\infty} \;\leq\; e^d \,\norm{w}_{\infty}.
\]
\end{lem}
\begin{proof}
There is nothing to prove if $\norm{f}_\infty = \norm{w}_\infty.$
Assume $\norm{f}_\infty > \norm{w}_\infty.$
Define a function $\Psi: \parenth{0, \norm{f}_{\infty}} \to \R^+$ as follows.
By \Href{Lemma}{lem:uniqueness_det_fixed_height_john}, for any 
$\alpha \in \parenth{0, \norm{f}_{\infty}}$, there is a solution $g_\alpha$ to Fixed-height John problem \eqref{eq:john_problem_pos_height} for $f$ and $w$ with $\xi = \alpha.$  Let 
\[
g_\alpha(x) = \tilde{\alpha} \, w \!\parenth{A_\alpha^{-1}\parenth{x - a_\alpha}}
\]
for some positive-definite $A_\alpha,$ point $a_\alpha,$ and $\tilde{\alpha} = \frac{\alpha}{\norm{w}_\infty}.$
Set 
\[
\Psi(\alpha) = \det A_{\alpha}.
\]
For any $\alpha_1, \alpha_2 \in \parenth{0, \norm{f}_{\infty}}$ and $\lambda \in [0,1],$ 
we claim
\begin{equation}
\label{eq:ellips-d_quasi_log-concavity}
\parenth{\Psi\!\parenth{\alpha_1^\lambda \,\alpha_2^{1-\lambda}}}^{1/d} 
\;\ge\;
\lambda \,\parenth{\Psi(\alpha_1)}^{1/d} 
\;+\; 
\parenth{1 - \lambda}\,\parenth{\Psi(\alpha_2)}^{1/d}.
\end{equation}
Indeed, by \Href{Proposition}{prp:inner-function-interpolation},
\[
\Psi\!\parenth{\alpha_1^\lambda \,\alpha_2^{1-\lambda}} 
\;\ge\;
\det\parenth{\lambda A_1 + \parenth{1-\lambda}A_2}.
\]
Now, \eqref{eq:ellips-d_quasi_log-concavity} follows immediately from
Minkowski's determinant inequality \eqref{eq:minkowski_det_ineq}.

Set 
\[
\Phi(t) 
= 
\parenth{\Psi\!\!\parenth{e^t}}^{1/d}
\]
for $t \in \parenth{-\infty, \ln \norm{f}_{\infty}}.$
By \eqref{eq:ellips-d_quasi_log-concavity}, $\Phi$ is concave on its domain.

Also, since $w$ is the solution to Functional John problem  \eqref{eq:john_func_problem}, for every $\alpha$ in the domain of $\Psi$, 
\[
\Psi(\alpha)\,\alpha 
\;\leq\;
\Psi\!\parenth{\norm{w}_{\infty}}\norm{w}_{\infty}.
\]
Letting   $t_0 = \ln  \norm{w}_{\infty}$, we obtain
\[
\Phi(t)
\;\leq\;
\Phi(t_0)\,e^{\frac{t_0 - t}{d}}
\]
for any $t$ in the domain of $\Phi.$ The right-hand side is a convex function in $t$, whereas $\Phi$ is concave. Since they agree at $t = t_0,$ we conclude that the graph of $\Phi$ lies below the tangent line to $\Phi(t_0)\,e^{\frac{t_0 - t}{d}}$ at the point $t_0.$ Thus
\[
\Phi(t)
\;\leq\;
\Phi(t_0)\,\!\!\parenth{1 - \frac{t - t_0}{d}}.
\]
As $t \to \ln \norm{f}_{\infty}$ and noting that $\Phi$ remains positive, we get
\[
0
\;\leq\;
1 - \frac{\ln \norm{f}_{\infty}}{d} + \frac{t_0}{d}.
\]
In other words,
\[
t_0 \;\ge\; -\,d + \ln \norm{f}_{\infty}.
\]
Hence, 
\[
\norm{w}_{\infty} 
=
e^{t_0}
\;\ge\; 
e^{-d}\,\norm{f}_{\infty}.
\]
This completes the proof of the lemma.
\end{proof}

\bibliographystyle{alpha}
\newcommand{\etalchar}[1]{$^{#1}$}

\end{document}